\documentclass[letterpaper,10pt]{article}

\usepackage{subfig}
\usepackage{graphicx}
\usepackage{amssymb}
\usepackage{amsthm}
\usepackage{amsmath}
\usepackage{epstopdf}
\usepackage{authblk}
\usepackage{cite}
\usepackage[margin=1.5in,vmargin=1.5in,includefoot]{geometry}
\usepackage{palatino}
\usepackage{microtype}

\usepackage{setspace}

\newtheorem{proposition}{Proposition}
\newtheorem{algorithm}{Algorithm}

\title{
\Large{Multiscale numerical methods for passive advection-diffusion in incompressible turbulent flow fields}}
\author[1]{Yoonsang Lee\thanks{ylee@cims.nyu.edu}}
\author[2]{Bjorn Engquist\thanks{engquist@math.utexas.edu}}
\affil[1]{Courant Institute of Mathematical Sciences, New York University}
\affil[2]{Department of Mathematics and ICES, University of Texas at Austin}
\date{}

\begin{document}
\maketitle

\begin{abstract}
We propose a seamless multiscale method which approximates the macroscopic behavior of the passive advection-diffusion equations with steady incompressible velocity fields with multi-spatial scales. The method uses decompositions of the velocity fields in the Fourier space, which are similar to the decomposition in large eddy simulations. It also uses a hierarchy of local domains with different resolutions as in multigrid methods. The effective diffusivity from finer scale is used for the next coarser level computation and this process is repeated up to the coarsest scale of interest.
The grids are only in local domains whose sizes decrease depending on the resolution level so that the overall computational complexity increases linearly as the number of different resolution grids increases. The method captures interactions between finer and coarser scales but has to sacrifice some of the interaction between the fine scales. The proposed method is numerically tested with 2D examples including a successful approximation to a continuous spectrum flow.
\end{abstract}

\textit{Keywords:} multiscale methods, seamless, advection enhanced diffusion

\section{Introduction}\label{sec:INTRO}
Multiscale problems in science and engineering pose severe challenges for numerical simulations. In a direct simulation the finest scales of interest must be well resolved over domains defined by the coarsest scale. For example, in the description of fluid flows, if the velocity field has high wavenumber component of order $\frac{1}{\epsilon}$ for $ 0<\epsilon\ll 1$,  a direct numerical simulation requires a mesh size at least of order $\epsilon$ in all dimensions. This makes the direct numerical simulation prohibitive for very small $\epsilon$ values over domains of size $\mathcal{O}(1)$. 

A number of computational techniques for multiscale problems have recently been suggested in which the microscales are only approximated on a set of small domains in order to reduce the computational cost. One common feature of these methods is that microscale simulations are coupled to a macroscale model that is defined for the full computational domain. Superparameterization which was proposed by Grabowski and further developed and analyzed by Majda and others (see \cite{CRCP,SP,SPperspective} and references therein) is an example which uses small local domains in space and time to reduce the computational complexity of the simulation of atmosphere.

The upscaling for reservoir simulation (see \cite{DurUPSCALING,DurBLOCKPERM} for example, and references therein) is another example of multiscale methods, which describe macroscopic behavior using local simulations. Upscaling is a technique for generating coarse scale models of highly heterogeneous subsurface formations using computations over microscale local domains.
In these upscaling techniques, any global fine scale calculations is not required but the local fine scale simulations cover the full domain.

The heterogeneous multiscale method (HMM) \cite{origHMM, generalHMM} is also an example which has similarity with Superparameterization and the upscaling. HMM is a general framework for designing multiscale algorithms, not restricted to a specific problem, and it has been used for various problems and applications (see \cite{generalHMM} and reference therein). In HMM the microscale simulations, which cover only part, are used to supply the missing microscale data to the macroscale system. The microscale computations are independent of each other which can be utilized to expedite the simulation using parallel computations but they are constrained to be consistent with the local macroscale state. 

One difficulty with all these methods is that they typically rely on scale separation in the original problem. The purpose of introducing and extending what we called seamless heterogeneous multiscale method (SHMM) in \cite{parabolicSHMM} is to develop a technique based on the HMM idea but to overcome the problem with scale separation and to allow for simulations when the scales are not known a priori. The method captures the interactions between finer and coarser scales without scale separation but has to sacrifice some of the interactions between fine scales.

More specifically, we consider a seamless method of the following passive advection-diffusion equation of a scalar field $u(t,\mathbf{x}):[0,\infty)\times \mathbb{R}^d\to\mathbb{R}$ in a given steady incompressible velocity field $\mathbf{v}(\mathbf{x}):\mathbb{R}^d\to\mathbb{R}^{d}$,
\begin{equation}\label{eq:model}
	\begin{split}
  \frac{\partial u(t,\mathbf{x})}{\partial t}&=\mathbf{v}(\mathbf{x})\cdot\nabla u+\kappa\Delta u, \qquad (t,\mathbf{x})\in (0,\infty)\times\mathbb{R}^d\\
	\nabla\cdot \mathbf{v}&=0
  	\end{split}
\end{equation}
where $\kappa>0$ is the molecular diffusivity. This equation often describes the flow of physical quantities, such as particles or heat, of a wide range of excited spatial and temporal scales. It has many important applications such as temperature prediction in fluid flows, and the transportation through heterogeneous porous media. In this study we include a steady case when $\mathbf{v}(\mathbf{x})$ is of turbulent nature.
Turbulent flows have a wide range of spatial scales and this multiscale property makes the model problem very complicated to analyze. Even the simulation of this linear model problem is challenging because of the wide range of active spatial and temporal scales of the velocity field. The small scale components of the velocity field usually shows dissipative behavior and thus the model is often model by an enhanced diffusivity (also called 'Eddy diffusivity' in turbulence literature \cite{TUB}). 

Our seamless method approximates the enhanced diffusion by accurately approximating the effective diffusivity through a hierarchy of local microscale simulations involving the different scale components of the velocity fields rather than by estimating it in an ad-hoc way. The seamless strategy is based on the HMM framework of focusing on the macroscale and have local microscale problems supplying approximations to the effective or homogenized diffusivity \cite{generalHMM}. It also contains components of multigrid (MG) and of large eddy simulation (LES). The SHMM technique will have a hierarchy of grids without scale separation as in MG. In contrast to MG, SHMM does not approximate the solution on the fine grids over the global domain. Instead, it has local domains whose size decreases depending on the scale level of the velocity component to have the same computational complexity at each level. Another key features of SHMM is that it requires filtering or projections of functions similar to that in LES in order to decompose the velocity field into different scale parts which can be well represented on appropriate grids. Once it has the decomposition of the velocity field, iterated technique is used to capture the block effective diffusivity using a combination of local simulations and a closed form of the effective diffusivity of the shear flows \cite{advdifreview}. 

The terminology 'seamless' for HMM has been already used in \cite{seamlessHMM}. The key goal of that method is to overcome the coupling between the macro and micro states of the system, which can be a difficult problem in actual implementations, particularly when the reconstruction of microscale states consistent with the macroscale state is not obvious. Essentially, this method still requires scale separation of the problem and we use the acronym SHMM here for the heterogeneous multiscale methods for problems with a wide range of scales without any scale gap. There are also other methods to calculate the large scale behavior of the passive advection-diffusion equation in turbulent velocity fields. One good example uses Monte-Carlo simulations and is further extended by hierarchical ways (\cite{waveMC,FourierMC} and references therein).

This paper has the following structure. In Section \ref{sec:REVIEW}, we review some results of homogenization theory of the model problem for periodic flows in addition to the closed form of the effective diffusivity of shear flows. The Heterogeneous Multiscale Method is also reviewed since it is used as the core component of the new method. SHMM is described in Section \ref{sec:SHMM} and the related decomposition of the velocity field and iteration process for effective diffusivity is explained in detail. The method is numerically tested for some numerical examples including shear flow, combinations of cellular flows, and continuous spectrum velocity field problem in Section \ref{sec:NUMERICAL}.

\section{Homogenization and Heterogeneous Multiscale Methods}\label{sec:REVIEW}
In this section, we recall some classical results of Homogenization theory \cite{BLP, RusHomo} for periodic cellular velocity fields (see \cite{PAVA,FanPa3,advdifreview} for random flows). The homogenization results show existence of macroscopic models of the advection-diffusion equation under periodic incompressible flows. This provides a model for determining the macroscopic solver in SHMM. We also review the closed formula for the effective diffusivity of the shear flow which is explicitly used in SHMM without numerical simulations to reduce the computational complexity. 
After homogenization, we review Heterogeneous Multisclae Mehtod (HMM) which is a building block of SHMM. 

\subsection{Homogenization of periodic flows}\label{homogenization}
Homogenization of periodic structures is a well studied field \cite{BLP, RusHomo} and we briefly review the homogenization theory of periodic mean-zero flows which provides a mascroscopic model of the passive scalar with an enhanced diffusivity at large scales and long times due to the combined effects of molecular diffusion and advection. 
In many applications, fluid flows are driven by a large-scale pressure gradient, and the resulting flow consists of some constant mean motions and fluctuations \cite{constmeanflow,meanflow}. We restrict our interest to zero-mean flows to discuss the diffusive long-time and large-scale behaviors. Homogenization theory of non-zero mean flows can be found in \cite{quasiperiodichomo,pavliotisthesis}.

\subsubsection{Homogenization with zero-mean periodic velocity field}
Let $\mathbf{v}(\mathbf{y})$ be an incompressible velocity field which is 1-periodic in $\mathbf{y}$ and has zero-mean velocity. 
We further assume that the initial value of $u$ are slowly varying, so that for a small positive parameter $\epsilon\ll1$, it can be modeled as
$$u(0,\mathbf{x})=g(\epsilon\mathbf{x}).$$
We then rescale the time and space according to diffusive scaling
$$\mathbf{x}\to\epsilon\mathbf{x},$$
$$t\to\epsilon^2 t.$$
Under this rescaling, the new passive scalar field $u^{\epsilon}(t,\mathbf{x})$ of \eqref{eq:model} satisfies the equation
\begin{equation}
	\begin{split}
	\frac{\partial u^{\epsilon}(t,\mathbf{x})}{\partial t}=&\frac{1}{\epsilon}\mathbf{v}(\frac{\mathbf{x}}{\epsilon})\cdot\nabla u+\kappa\nabla u,\\
	u^{\epsilon}(0,\mathbf{x})=&g(\mathbf{x}).
	\end{split}
\end{equation}
Using a stream function $\psi(\mathbf{x/\epsilon})$ of the velocity field
$$\mathbf{v}=(-\partial_{x_2}\psi,\partial_{x_1}\psi),$$
the rescaled equation is given by
\begin{equation}
	\begin{split}
	\frac{\partial u^{\epsilon}(t,\mathbf{x})}{\partial t}=&\nabla\cdot\left(\begin{pmatrix}\kappa&-\psi(\mathbf{x/\epsilon})\\\psi(\mathbf{x/\epsilon})&\kappa\end{pmatrix}\nabla u^{\epsilon}\right)\\
	u^{\epsilon}(0,\mathbf{x})=&g(\mathbf{x}).
	\end{split}
\end{equation}

For a long-time and large space-scale (which is implied by $0<\epsilon<\ll 1$), the rescaled scalar field is approximated by $U(t,\mathbf{x})$, which is the solution of the homogenized equation
\begin{equation}\label{eq:homo}
	\begin{split}
	\frac{\partial U(t,\mathbf{x})}{\partial t}=&\nabla \left(\mathcal{K}\nabla U\right)\\
	U(0,\mathbf{x})=&g(\mathbf{x})
	\end{split}
\end{equation}
with constant, symmetric positive definite diffusivity $\mathcal{K}$ defined by
$$\mathcal{K}=\kappa I+\frac{1}{2}\int_{[0,1]^2}\left(\mathbf{v}(\mathbf{y})\otimes\chi(\mathbf{y})+\chi(\mathbf{y})\otimes \mathbf{v}(\mathbf{y})\right)d\mathbf{y}$$
where $I$ is the identity matrix and the vector field $\chi(\mathbf{v})$ is the solution to the cell problem
\begin{equation}\label{eq:cell}
-\mathbf{v}(\mathbf{y})\cdot\nabla \chi-\kappa\nabla \chi=\mathbf{v}(\mathbf{y})\end{equation}
with periodic boundary condition and $\int_{T^2}\chi(\mathbf{y})d\mathbf{y}=0$.

The convergence of $u^{\epsilon}$ to $U$ is made precise in the following sense \cite{BLP, RusHomo,STUART}
$$\lim_{\epsilon\to 0}\sup_{0\leq t\leq t_0}\sup_{\mathbf{x}\in\mathbb{R}^d}|u^{\epsilon}(t,\mathbf{x})-U(t,\mathbf{x})|=0$$
for every finite time $t_0$ which is independent of $\epsilon$.
The homogenization result says that under periodic flows with mean zero, the long-time and large-scale behavior of the advection-diffusion equation can be effectively described by an enhanced diffusivity, which is not necessarily isotropic. Thus, the main focus of our approach in this study is the estimation of the effective diffusivity from the multiscale velocity field without solving \eqref{eq:cell}.

\subsubsection{Homogenization of Shear Flow}
The seamless method proposed in this paper decomposes the velocity field into different scale parts and some of them are approximately shear flows. We utilize a simple formula for the effective diffusivity of shear flows to accelerate computations. In this section, we review some results on the homogenization of shear flow (see \cite{STUART} for the derivation) with an anisotropic molecular diffusivity $K=\begin{pmatrix}\kappa_1&0\\0&\kappa_2\end{pmatrix}$. 

Let the velocity field $\mathbf{v}(\mathbf{y})$ is given by
$$\mathbf{v}=(0,v_2(y_1))$$
where $v_2(y_1)$ is a smooth, 1-periodic function with mean zero. Then the effective diffusivity $\mathcal{K}$ by the homogenization theory is given by the following $2\times2$ matrix
$$\mathcal{K}=\begin{pmatrix}\kappa&0\\0&\mathcal{K}_{22}\end{pmatrix}$$
where
$$\mathcal{K}_{22}=\kappa_2+\frac{1}{\kappa_1}\int_0^1|\frac{d\phi}{dy_1}|^2dy_1$$
and $\phi$ is a periodic solution to
$$-\frac{d^2\phi(y_1)}{dy_1^2}=v_2(y_1).$$
If $\psi$ is the stream function of $\mathbf{v}$ with mean zero, we have
\begin{equation}\label{eq:shearhomoX}\mathcal{K}_{22}=\kappa_2+\frac{1}{\kappa_1}\int_0^1|\psi|^2dy_1.\end{equation}

Similarly, in the case of $\mathbf{v}(\mathbf{y})=(v_1(y_2),0)$, we have
$$\mathcal{K}=\begin{pmatrix}\mathcal{K}_{11}&0\\0&\kappa_2\end{pmatrix}$$
where
\begin{equation}\label{eq:shearhomoY}\mathcal{K}_{11}=\kappa_1+\frac{1}{\kappa_2}\int_0^1|\psi|^2dy_2\end{equation}
for the zero-mean stream function $\psi$ of $\mathbf{v}(\mathbf{y})$.

\subsection{Heterogeneous Multiscale Method for Homogenization Problems}\label{HMM}
The Heterogeneous Multiscale Method is a general framework for multiscale problems with focus on the macroscopic behavior of the solution rather than describing all microscopic details and with computational complexity independent of the finest scale \cite{FDHMM,generalHMM}. By assuming existence of a macroscopic model of the problem
\begin{equation}U_t=F(U)\label{HMMmodel}\end{equation}
with unknown macroscopic force $F(\cdot)$, the main process of HMM is to approximate $F(U)$ by solving local microscale problems on the fly.

The basic components of HMM are as follows:
\begin{enumerate}
\item \textit{A macroscopic solver}. Based on knowledge of the macroscale behavior of the problem, we make an assumption about the structure of the macroscopic model, for which we select a suitable macroscale solver. 
\item \textit{Estimation of the missing macrscale force $F(U)$} 
	\begin{enumerate}
		\item \textit{Constrained microscale simulation}. At each point where macroscale data are required, perform a series of microscopic simulations on local grids which are constrained so that they are consistent with the local macroscale data.
		\item \textit{Data processing}. Use the results from the microscopic simulations to
extract the macroscale data required in the macroscale solver.
	\end{enumerate}	
\end{enumerate}

We now state the HMM procedure for our model problem \eqref{eq:model}. For our model problem, the effective force $F(U)$ is given by 
\begin{equation}
F(U)=\nabla\cdot P(U)
\end{equation}
where $P(U)$ is the effective flux. Thus the main focus is to approximate the effective flux $P(U)$ from local microscale solutions.

\medskip
\textbf{HMM algorithm for \eqref{eq:model}}

At the $n$-th macro time step with a macro solution $U^n$ defined on macro grid points $\{X_{ij},i,j=1,2,...,N\}$:
\begin{enumerate}
	\item Given the current state of the macro variables $U^n$, re-initialize the micro variable for each local domain $\Omega_{ij}$:
	$$u^{n,0}_{ij}=R_{ij}U^n$$
	where $R_{ij}$ is a reconstruction operator which plays the same role as the interpolation or prolongation operators in the multigrid method. In general, the reconstruction operators are independent of the macro grid points, that is, $R_{ij}=R$ for all $i$ and $j$.
	\item  Evolve microscale problems to reach a quasi-stationary state (that is, $\partial_t u=\delta\ll 1$) with a microscopic time step $\delta t$ and a fine grid:
		$$u^{n,m+1}_{ij}=\mathcal{S}^{\delta t}_{ij}(u^{n,m};U^n),\quad m=0,...,M-1,\mbox{ in }\Omega_{ij}.$$
	where $\mathcal{S}^{\delta t}_{ij}$ is the microscale solution operator defined in $\Omega_{ij}$; this is also dependent on $U^n$ through the constraints. For the boundary, a periodic boundary condition, $u^{n,m}-u^{n,0}=u^{n,m}-RU^n$ is $\Omega_{ij}$ periodic, is used. 
	
	\item Estimate the effective flux
	$$P(U)_{ij}=\mathcal{D}_{ij}(u^{n,0}_{ij},u^{n,1}_{ij},...,u^{n,M}_{ij})$$
	where $\mathcal{D}_{ij}$ is some data processing operator, which in general involves spatial and temporal averaging; this is sometimes referred to as the data estimator.
	\item Evolve the macro variables for one macro time step $\Delta t$ using the macro solver $S^{\Delta t}$:
	$$U^{n+1}=S^{\Delta t}(U^n;\{P(U)_{kl}\})$$
\end{enumerate}
In \cite{FDHMM}, it is shown for the periodic homogenization case the effective flux by the above HMM procedure approximates the homogenized flux with an approximation error of order $\mathcal{O}(\delta)$. 

If the velocity $\mathbf{v}(\mathbf{x})$ is well scale-separated with a periodicity $\epsilon$, that is, $\mathbf{v}(\mathbf{x})=\mathbf{v}(\mathbf{x}/\epsilon), 0<\epsilon\ll 1$, the size of the local domain $\Omega_{ij}$ is usually chosen to be a larger than a few multiples of the periodicity $\epsilon$ but still smaller than the macroscale grid spacing. For general flows which contains a wide range of scales, HMM requires local domains as large as the macroscale spacing in which they cover the whole domain with the fine grid. Also, if there is much finer scales than the local domain size, that is $\mathbf{v}(\mathbf{x})=\mathbf{v}(\mathbf{x}/\epsilon,\mathbf{x}/\epsilon^2)$, the local microscale simulations are also another multiscale problems which requires much finer grid points. 

For the passive advection-diffusion equation, as the homogenization theory implies, the effective flux $P(U)$ is a linear map of $\nabla U$, $\mathcal{K}\nabla U$. Thus for multiscale velocity field without scale separation, the main focus is how to approximate the effective diffusivity without using the finest grid to resolve the velocity field in the local domains covering the whole domain. One of the ideas of the seamless method descried in the next section is to use iterated homogenization after decompositions of the velocity fields.
The effective diffusivity of a single component is approximated by the matrix representation of the effective flux map from applying the HMM technique. More precisely, we solve the microscale problems with two different reconstructions so that the two different initial conditions have non-parallel gradients. To prevent ill-conditioning of the inversion of the effective flux map, we choose two reconstructions so that the initial conditions have gradients only in $x$ or $y$ directions (that is, the gradients of the two initial conditions are orthogonal). This process, which is sometimes called numerical homogenization, is a building block of our seamless method and we denote the effective diffusivity from a velocity field $\mathbf{v}$ with a base diffusivity matrix $K$ in a local domain $D$  by $\mathcal{K}[\psi,K;D]$. For simple exposition of the seamless algorithm, we will describe the method for the stream function $\psi$ of a zero-mean velocity field. 

\section{Seamless Heterogeneous Multiscale Method (SHMM)}\label{sec:SHMM}
In this section we describe the Seamless Heterogeneous Multiscale Method (SHMM) for \eqref{eq:model} without scale separation in the stream function and thus the velocity field. As in HMM, SHMM uses a macroscopic discretization $\{X_{ij}\}$ to represent the large-scale solution. To evolve the large-scale solution, SHMM calculates the effective flux in local domains $\Omega_{ij}$ centered at $X_{ij}$ with vertices at $X_{i\pm1,j\pm1}$ (see Figure \ref{fig:localdomains}). Up to this point, the only difference between the discretization of HMM and SHMM is the local domains in which for SHMM they overlap each other and cover the whole domain. 

To reduce the computational cost SHMM uses a hierarchy of different grid points along with decomposition of the stream function in the Fourier domain so that different grid points can be used to resolve different components of the stream function. The number of levels of local domains (or the first fine level) $\Omega^1_{ij}=\Omega_{ij}$ which is as large as the original local domain uses grids that are finer than the global macro grid but typically coarser than the finest scale of the velocity field. For the next level local domain $\Omega^2_{ij}$, we decrease the grid spacing so that finer scales can be resolved. We also decrease the domain size by the same factor as the grid spacing so that it has the same computational complexity of the first level discretization in $\Omega_1$ (we keep the center of the local domains at $X_{ij}$). The level of local domains is obviously determined by the finest scale of the velocity field (see Figure \ref{fig:localdomains} for three different levels of local domains at $X_{ij}$).
\begin{figure}[h]
	\centering
	\includegraphics[width=1\textwidth]{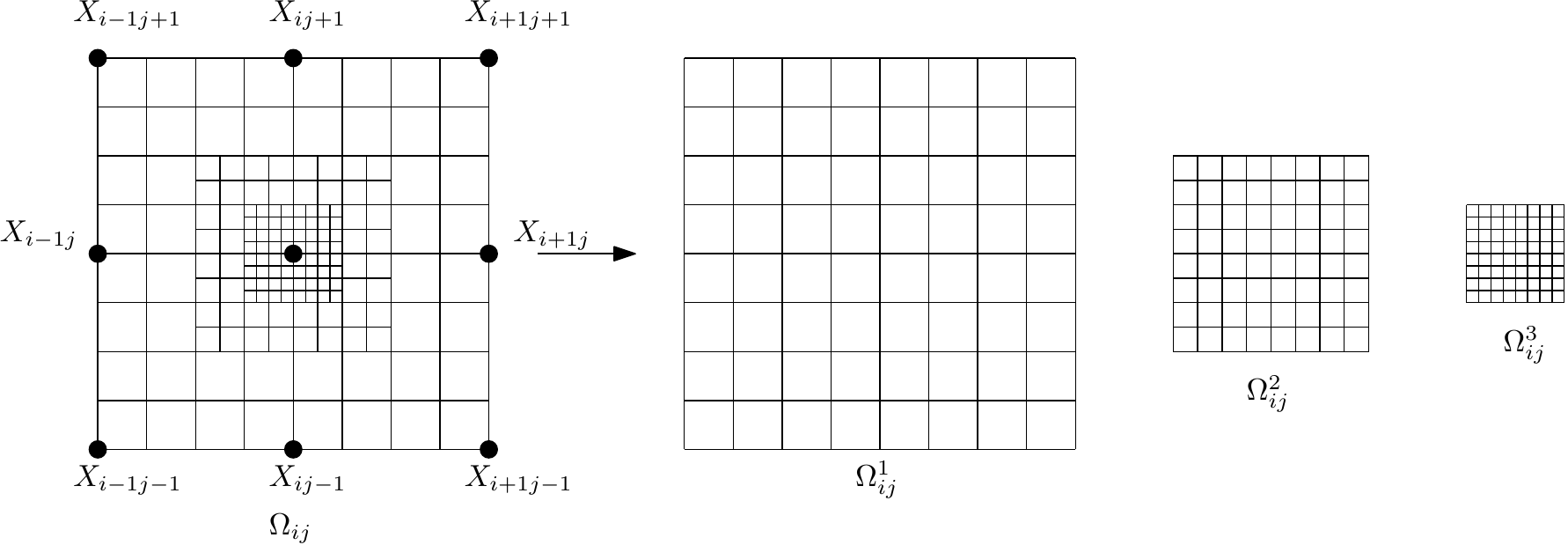}
	\caption{Hierarchy of local domains to resolve different scale components of stream function.}
	\label{fig:localdomains}
\end{figure}

Under the above setting we will focus on calculating the the effective diffusivity in a local domain $\Omega_{ij}$ centered at a macro grid point $X_{ij}$. In other words, using the notation in the last section, we focus on calculating $\mathcal{K}[\psi,\kappa;\Omega_{ij}]$. First we show the decomposition of the stream function into different scale components in Section \ref{subsec:decomposition}. In Section \ref{subsec:shmmdetail}, the seamless algorithm is presented to calculate the effective diffusivity $\mathcal{K}[\psi,\kappa;\Omega_{ij}]$ using hierarchical calculation of effective diffusivities from each component of the stream function. For simplicity of notations, we will suppress the indices $ij$ hereafter.

\subsection{Decomposition of stream function}\label{subsec:decomposition}
Different wavenumbers require different resolutions to resolve them with comparable accuracy. Thus the basic idea is to decompose the stream function in Fourier space. Let us point out another motivation for the decomposition in Fourier space. One of the key ideas of SHMM is to use iterated homogenization of well-separated scale components. In general, the well-separated condition is not applicable to continuous spectrum stream functions. But for shear flows, the following proposition implies that decomposition in the Fourier domain can impose the effect of scale separation so that they can be treated as well-separated in iterated homogenization.

\begin{proposition}\label{prop:ADVDIF:l2ortho}
Let $\mathbf{v_1}$ and $\mathbf{v_2}$ are two mean zero shear velocity fields in the same direction, say $x_1$, whose mean zero stream functions $\psi_1$ and $\psi_2$ are orthogonal in $L^2([0,1])$. Then the effective diffusivity from $\mathbf{v}_1+\mathbf{v}_2$ is equal to the iterated effective diffusivity by assuming that one of $\mathbf{v}_i,i=1,2$ has well-separated smaller scales than the other one. That is,
\begin{equation}
\mathcal{K}[\psi_1+\psi_2,\kappa]=\mathcal{K}[\psi_1,\mathcal{K}[\psi_2,\kappa]]=\mathcal{K}[\psi_2,\mathcal{K}[\psi_1,\kappa]]
\end{equation}
\end{proposition}
\begin{proof}
From \eqref{eq:shearhomoX},
$$\mathcal{K}[\psi_2,\kappa]=\begin{pmatrix}\kappa_1&0\\0&\kappa_2+\frac{1}{\kappa_1}\int|\psi_2|^2dx_1\end{pmatrix}$$
We apply (\ref{eq:shearhomoX}) with $\mathcal{K}[\psi_2,\kappa]$ as the base diffusivity which yields
$$\mathcal{K}[\psi_1,\mathcal{K}[\psi_2,\kappa]]=\begin{pmatrix}\kappa_1&0\\0&\kappa_2+\frac{1}{\kappa_1}\int|\psi_2|^2dx_1+\frac{1}{\kappa_1}|\psi_1|^2dx_1\end{pmatrix}=\mathcal{K}[\psi_2,\mathcal{K}[\psi_1,\kappa]].$$
On the other hand, from the orthogonality of $\psi_1$ and $\psi_2$ in $L^2$, we have
\begin{equation}
\begin{split}
\mathcal{K}[\psi_1+\psi_2,\kappa]=&\begin{pmatrix}\kappa_1&0\\0&\kappa_2+\frac{1}{\kappa_1}\int|\psi_1+\psi_2|^2dx_1\end{pmatrix}\\
=&\begin{pmatrix}\kappa_1&0\\0&\kappa_2+\frac{1}{\kappa_1}\int|\psi_2|^2dx_1+\frac{1}{\kappa_1}|\psi_1|^2dx_1\end{pmatrix}\\
=&\mathcal{K}[\psi_1,\mathcal{K}[\psi_2,\kappa]]
\end{split}
\end{equation}
\end{proof}

Based on the above proposition, we decompose the stream function in Fourier space with a decomposition factor $\alpha\in\mathbb{N}$. First, we decompose the stream function $\psi$ into two parts $\psi_1$ and $\psi_2$ in Fourier space
\begin{equation}
\psi=\psi_1+\psi_2
\end{equation}
where $\psi_2$ has only the wave numbers larger than $\alpha$ in both direction (see Figure \ref{fig:ADVDIF:decomp1})
$$\hat{\psi}_{2}(\mathbf{k})=\left\{\begin{array}{ll}\hat{\psi}(\mathbf{k})&\textrm{ if }|k_1|,|k_2|> \alpha\\0&\textrm{ otherwise}\end{array}\right.$$
and thus we use a fine grid in a local domain to resolve this component.
\begin{figure}[h]
	\centering
	\subfloat[Decomposition of the stream function into two parts\label{fig:ADVDIF:decomp1}]{\includegraphics[width=.45\textwidth]{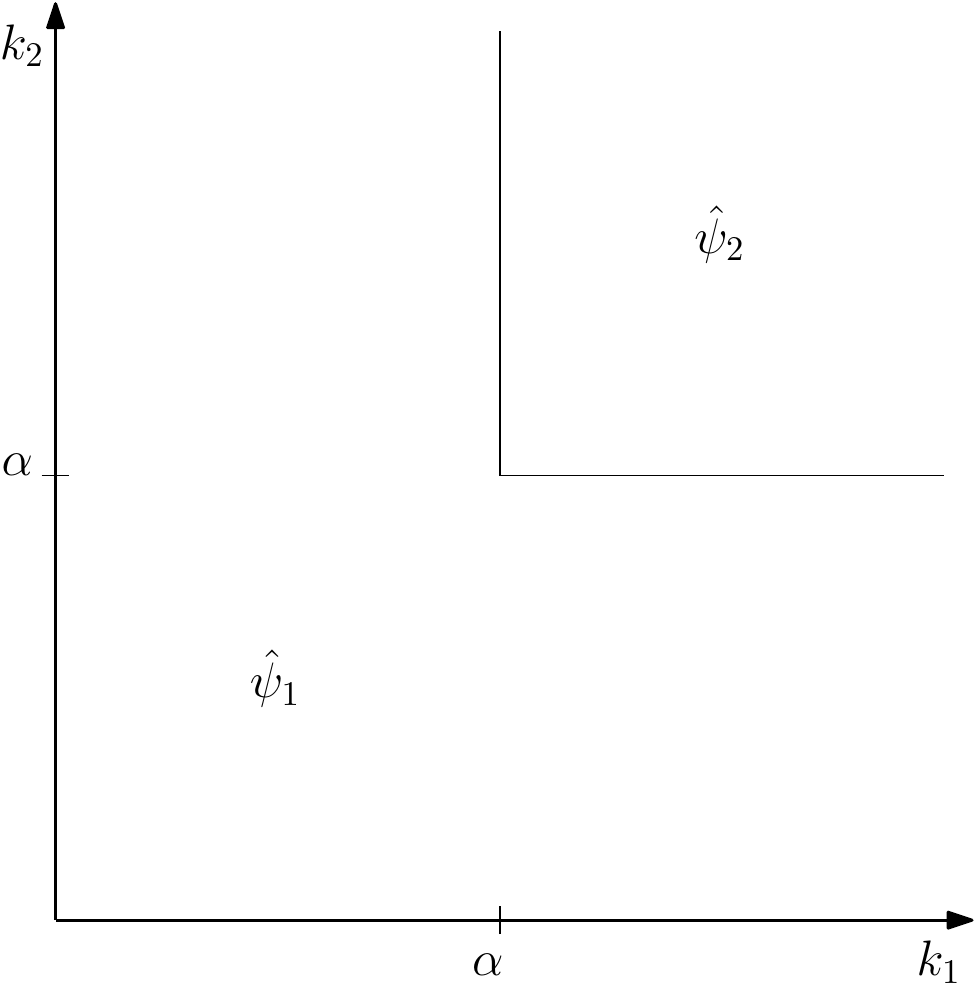}}
        \subfloat[Decomposition of $\psi_{1}$ into three parts\label{fig:ADVDIF:decomp2}]{\includegraphics[width=.45\textwidth]{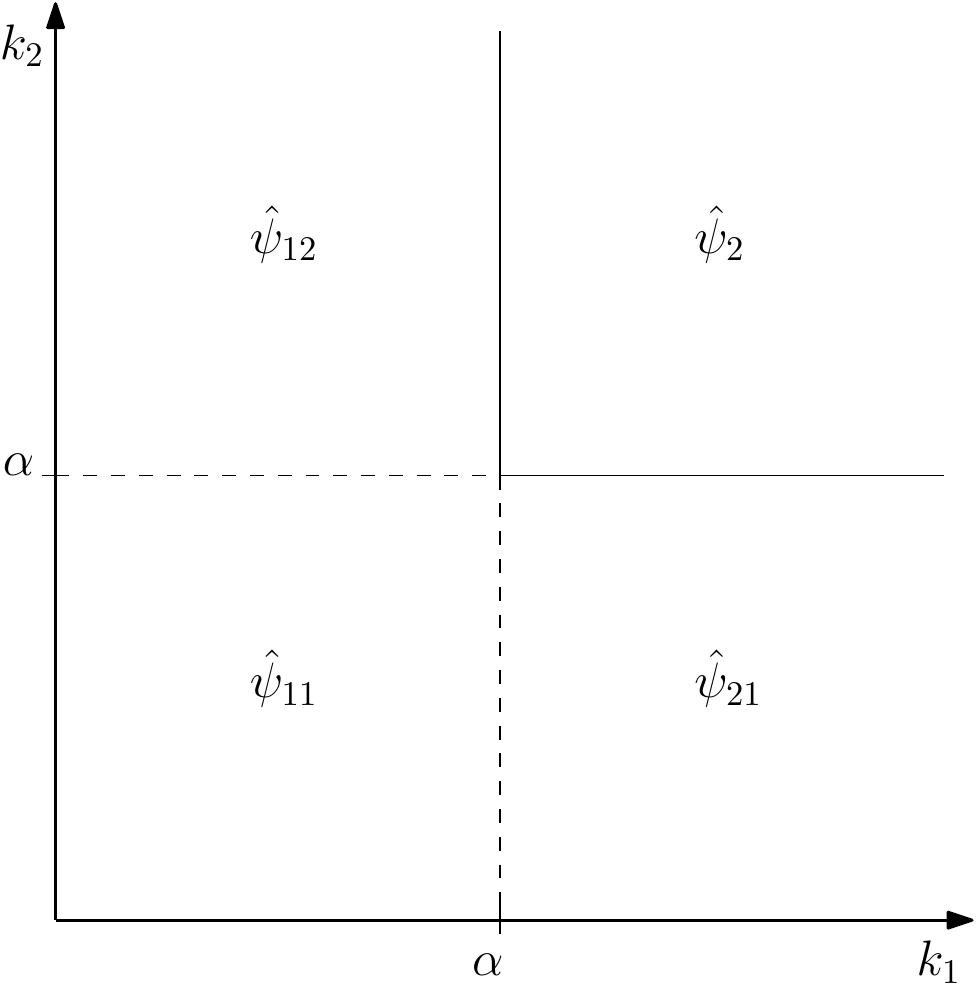}}
        \caption{Decomposition of the stream function in the Fourier domain.}\label{fig:ADVDIF:decomp}
\end{figure}
 On the other hand, $\psi_1$ contains low and high wave numbers in each direction and thus we need to decompose $\psi_1$ further so that each component can be resolved efficiently. 
 $\psi_{1}$  is decomposed into three different parts so that one part contains only the low wave numbers and the other parts contain high wave numbers only in one direction (see Figure \ref{fig:ADVDIF:decomp2}), that is,
$$\psi_1=\psi_{11}+\psi_{12}+\psi_{21}$$
where
\begin{align}\hat{\psi}_{11}(\mathbf{k})=&\left\{\begin{array}{ll}\hat{\psi}(\mathbf{k})&\textrm{if }|k_1|,|k_2|\leq\alpha\\0&\textrm{otherwise}\end{array}\right.\\
\hat{\psi}_{12}(\mathbf{k})=&\left\{\begin{array}{ll}\hat{\psi}(\mathbf{k})&\textrm{if }|k_1|\leq\alpha,|k_2|> \alpha\\0&\textrm{otherwise}\end{array}\right.\\
\hat{\psi}_{21}(\mathbf{k})=&\left\{\begin{array}{ll}\hat{\psi}(\mathbf{k})&\textrm{if }|k_1|> \alpha,|k_2|\leq \alpha\\0&\textrm{otherwise}\end{array}\right.
\end{align}
The coarse diagonal component, $\psi_{11}$, has only low wave numbers and thus can be well resolved using a coarse grid.
On the other hand, the two off-diagonal components, $\psi_{21}$ and $\psi_{12}$, have high wave numbers only in one direction $x_1$ and $x_2$ respectively and thus require a fine grid only in one direction. But the off-diagonal parts are not resolved using find grids. Instead, SHMM, for additional computational savings, uses the analytic formula \eqref{eq:shearhomoX} and \eqref{eq:shearhomoY}, the effective diffusivity of shear flows, by treating the off-diagonal components as shear flows. This process can be repeated recursively for the diagonal component. If $\psi_{2}$ has wave numbers larger than $\alpha^2$, we repeat the same decomposition process as we do for $\psi$ with the same decomposition factor $\alpha$ (see Figure \ref{fig:ADVDIF:decomp3} for iterated decomposition of the stream function). 
\begin{figure}
	\centering
	\includegraphics[width=.45\textwidth]{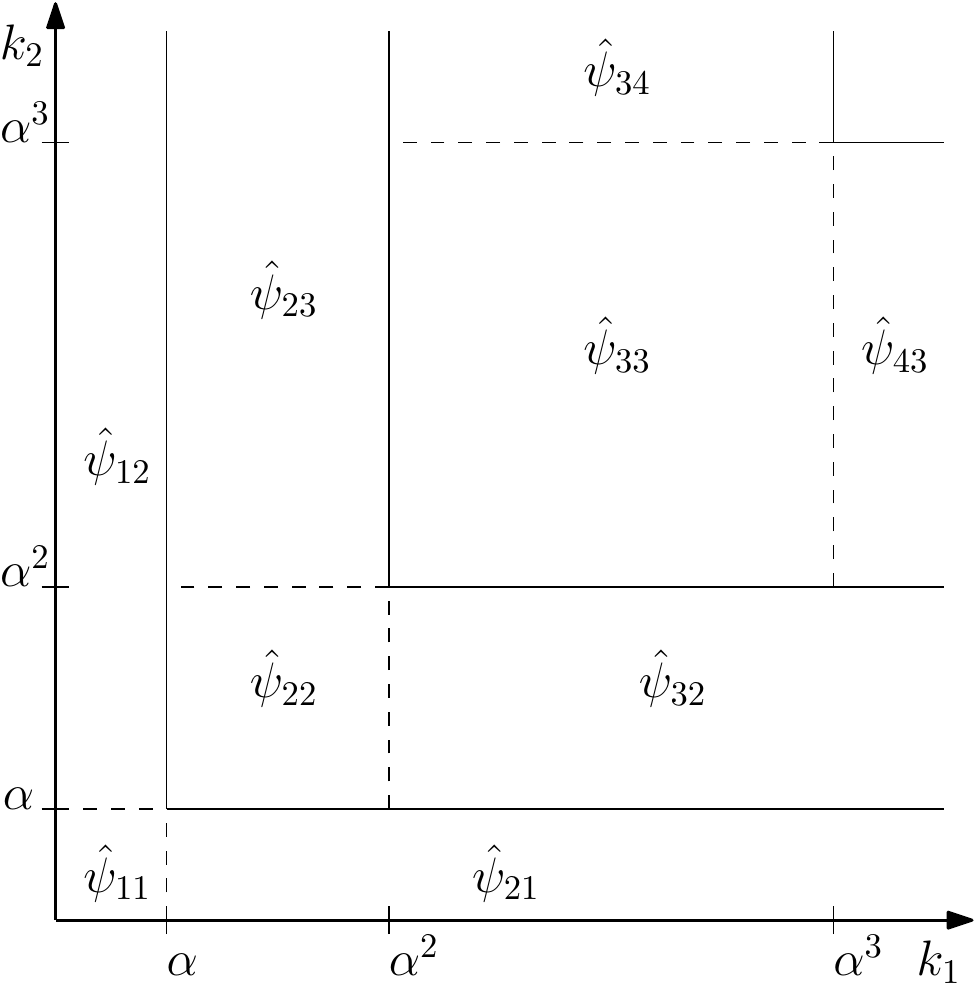}
        \caption{Repeated decomposition of the diagonal velocity fields.}\label{fig:ADVDIF:decomp3}
\end{figure}

\subsection{Hierarchical Calculation of the Effective Diffusivity}\label{subsec:shmmdetail}
Without loss of generality, we consider  a seamless algorithm for calculating the effective diffusivity from a two level decomposition of the stream function $\psi=\psi_{11}+\psi_{12}+\psi_{21}+\psi_{22}$. 
The main idea of the algorithm is to use the iterated homogenization from the finest scale to the largest scale. For example, if the stream function has only the diagonal components, $\psi_{11}$ and $\psi_{22}$, we first calculate the effective diffusivity from $\psi_{22}$, $\mathcal{K}^*=\mathcal{K}[\psi_{22},\kappa;\Omega_2]$ and then the total effective diffusivity of $\psi$, $\mathcal{K}[\psi,\kappa;\Omega_1]$ is given by $\mathcal{K}[\psi_{11},\mathcal{K}^*;\Omega_1]$. Unfortunately, the decomposition does not provide totally independent components. The off-diagonal components, $\psi_{12}$ and $\psi_{21}$, have high wave numbers interacting with $\psi_{22}$. We sacrifice the full interactions between the fine scales of the off-diagonal and diagonal components and instead capture weak interaction between them. 

The idea is to use the effective diffusivity $\mathcal{K}[\psi_{22},\kappa]$ from $\psi_{22}$ as the base diffusivity in the calculation of $\psi_{11}+\psi_{12}+\psi_{21}$ with an preprocessing step in the calculation of the effective diffusivity from $\psi_{22}$. We account for the effect of the off-diagonal components to $\psi_{22}$ by setting $\mathcal{K}[\psi_{12}+\psi_{21},\kappa]$ as the base diffusivity, that is, $\mathcal{K}[\psi_{22},\kappa]$ is replaced by $\mathcal{K}[\psi_{22},\mathcal{K}[\psi_{12}+\psi_{21},\kappa]]$

Before we discuss how to calculate the effect of $\psi_{22}$, we first explain how to calculate $\mathcal{K}[\psi_{12}+\psi_{21},\kappa]$ (which can be easily extended to the case of $\mathcal{K}[\psi_{11}+\psi_{12}+\psi_{21},\kappa]$). To expedite the calculation of $\mathcal{K}[\psi_{12}+\psi_{21},\kappa]$, we use the effective diffusivity of $\psi_{12}$ and $\psi_{21}$ separately. With this separation between the off-diagonal components, we use the analytic formula \eqref{eq:shearhomoX} and \eqref{eq:shearhomoY} by treating $\psi_{21}$ and $\psi_{12}$ as stream functions of shear flows.
More specifically, in the calculation of $\mathcal{K}[\psi_{12}]$, for example, we first use the analytic formula \eqref{eq:shearhomoY} for all vertical lines passing through each coarse grid points in $\Omega_1$ to homogenize the high wave numbers in the $x_2$ direction. This procedure gives effective diffusivity as a function of $x_1$ only, say $\mathcal{K}^{*}(x_1)$. Thus, the next homogenization in the $x_1$ direction is given by the reciprocal of the harmonic average $\left(\int \frac{1}{\mathcal{K}^{*}(x_1)}\right)^{-1}$ which comes from the homogenization of layered material (see \cite{BLP} or \cite{RusHomo}).

\begin{algorithm}{Effective diffusivity from each off-diagonal component}
\label{alg:offdiagonal}
\end{algorithm}
Without loss of generality, we consider $\mathcal{K}[\psi_{12},\kappa]$.
\begin{enumerate}
\item For each vertical line passing spaced by the coarse mesh size in $\Omega_1$, calculate the effective diffusivity using \eqref{eq:shearhomoY} which yields effective diffusivity $\mathcal{K}^{*}(x_1)$ as a function of $x_1$.
\item The effective diffusivity of $\psi_{12}$ is given by 
$$\mathcal{K}[\psi_{12},\kappa]=\left(\int \frac{1}{\mathcal{K}^{*}(x_1)}\right)^{-1}.$$
\end{enumerate}

Both $\mathcal{K}[\psi_{12},\kappa]$ and $\mathcal{K}[\psi_{21},\kappa]$ have effects from $\kappa$ in addition to the net effect from $\psi_{12}$ and $\psi_{21}$. Thus we subtract the repeated effect from $\kappa$ and the total effective diffusivity from $\psi_{12}$ and $\psi_{21}$ is given by the following algorithm

\begin{algorithm}{Effective diffusivity from the all off-diagonal components, $\mathcal{K}[\psi_{12}+\psi_{21},\kappa]$}
\label{alg:sumofoffdiagonal}
\end{algorithm}
\begin{enumerate}
\item Calculate $\mathcal{K}_{12}=\mathcal{K}[\psi_{12},\kappa]$ and $\mathcal{K}_{21}=\mathcal{K}[\psi_{21},\kappa]$ using Algorithm \ref{alg:offdiagonal}.
\item $\mathcal{K}[\psi_{12}+\psi_{21},\kappa]$ is given by the sum of the effects from $\psi_{12}$ and $\psi_{21}$ subtracted by the common base diffusivity $\kappa$
\begin{equation}
\mathcal{K}[\psi_{12}+\psi_{21},\kappa]=\mathcal{K}[\psi_{12},\kappa]+\mathcal{K}[\psi_{21},\kappa]-\kappa I_2
\end{equation}
where $I_2$ is the $2\times 2$ identity matrix.
\end{enumerate}

Let us now describe the comprehensive algorithm to calculate the effective diffusivity from all components $\psi=\psi_{11}+\psi_{12}+\psi_{21}+\psi_{22}$. As mentioned above, we use $\mathcal{K}[\psi_{12}+\psi_{21},\kappa]$ as the base diffusivity in the calculation of the effective diffusivity from $\psi_{22}$ to account for the effect of the off-diagonal components.
\begin{algorithm}{Effective diffusivity from all components, $\mathcal{K}[\psi;\Omega_1]$}\label{alg:total}
\end{algorithm}
\begin{enumerate}

\item Calculate $\mathcal{K}_{off}=\mathcal{K}[\psi_{12}+\psi_{21},\kappa;\Omega_1]$ using Algorithm \ref{alg:sumofoffdiagonal}.

\item Calculate $\mathcal{K}^*=\mathcal{K}[\psi_{22},\mathcal{K}_{off};\Omega_2]$ with a base diffusivity $\mathcal{K}_{off}$ from the previous step.

\item Let $\mathcal{K}^{*}_{net}$ be given by
$$\mathcal{K}^{*}_{net}=\mathcal{K}^*-\mathcal{K}_{off}+\kappa I_2$$
which has the effect from $\psi_{22}$ without the effects from $\psi_{12}$ and $\psi_{21}$.
The total effective diffusivity $\mathcal{K}[\psi_1+\psi_2,\kappa]$ is given by
\begin{equation}
\mathcal{K}[\psi,\kappa;\Omega_1]=\mathcal{K}[\psi_{11}+\psi_{12}+\psi_{21},\mathcal{K}^{*}_{net};\Omega_1]
\end{equation}
\end{enumerate}
$\mathcal{K}^*$ in the third step of Algorithm \ref{alg:total} is to include the effect of $\psi_{22}$ in the calculation of the effective diffusivity of the off-diagonal components. One may consider the reversed order to capture the (not necessarily full) interactions between the off-diagonal and diagonal $\psi_{22}$ components. But this approach requires two calculations of the effective diffusivity from $\psi_{22}$ with two different base diffusivities. Because $\psi_{22}$ requires a simulation using grids, this approach has computational complexity that increases exponentially as the number of decomposition level increases.

\section{Numerical Examples}\label{sec:NUMERICAL}
We test SHMM for the passive advection-diffusion equation in various velocity fields with mean zero. A discontinuous initial condition (Figure \ref{fig:init}) is used for all tests with Dirichlet and periodic boundary conditions in $x_1$ and $x_2$ directions respectively:
\begin{equation}
\frac{\partial u^{\epsilon}(t,\mathbf{x})}{\partial t}+\mathbf{v}^{\epsilon}(\mathbf{x})\cdot\nabla u=\Delta u, \quad \mathbf{x}=(x_1,x_2)\in \Omega=[0,1]^2
\end{equation}
$$u(\mathbf{x},0)=\left\{\begin{array}{l}0,\quad x_1>\frac{1}{2}\\1\quad x_1\leq \frac{1}{2}\end{array}\right.$$
$$u(\mathbf{x},t)=\left\{\begin{array}{l}0,\quad x_1=1\\1\quad x_1=0\end{array}\right.$$
$$\textrm{periodic boundary condition in $y$-direction}.$$

\begin{figure}
\centering
\includegraphics[width=8cm]{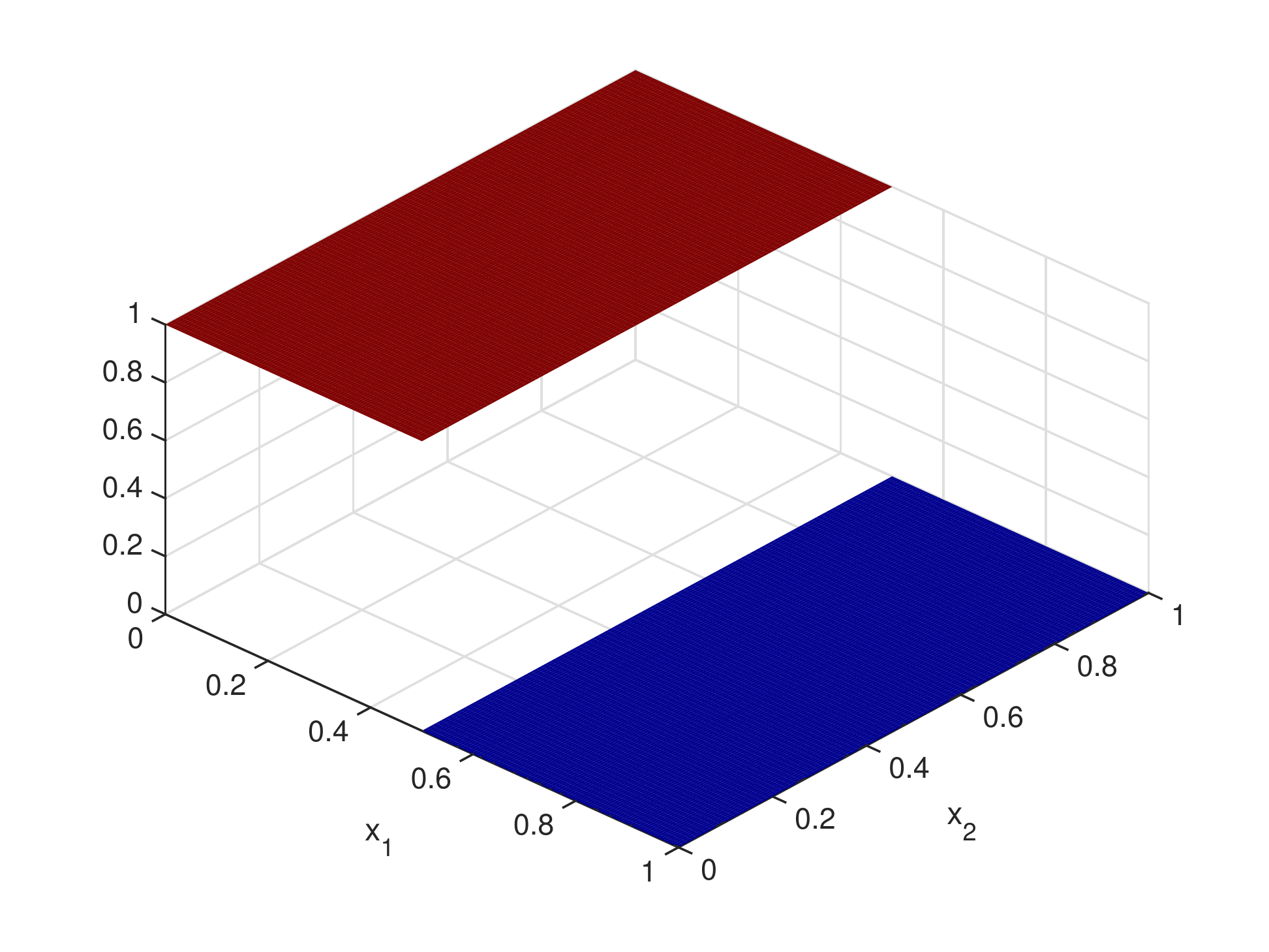}
\caption{Initial profile of $u$ for all tests.}\label{fig:init}
\end{figure}
As a reference solution to compare with SHMM solutions, direct numerical solutions (DNS) are computed using a second order finite difference scheme on $512 \times 512$ grid points in space and a second order Runge-Kutta with a time step $\Delta t=7.5\times 10^{-7}$. The 512 grid points guarantee that there are more than 10 grid points per wavelength even for the largest wavenumber of the stream function in the following tests. The solution from the pure diffusion without the velocity field (which is labeled as `No Advection' in all the numerical results below) are provided as a baseline to check the enhanced diffusion due to advection.

\subsection{Random shear flow}
The first test problem is a random shear flow. The stream function is randomly generated in $x_2$ direction with a maximum wavenumber 50. 
For shear flows, SHMM uses the analytic formula for the effective diffusivity and thus this experiment shows the validity of the homogenization theory. Using the analytic formula \eqref{eq:shearhomoY} the effective diffusivity is an anisotropic given as
\begin{equation}
\mathcal{K}[\psi,1]=\begin{pmatrix}
2.9994&0\\
0&1\\
\end{pmatrix}
\end{equation}
Figure \ref{fig:exp1b} shows the SHMM solution with this effective diffusivity at $t=0.1$ along with DNS solution and another DNS solution without advection (that is, without stream function). The SHMM solution is on top of the DNS solution while the pure diffusion without advection exhibits significantly less diffusion than the advection enhanced diffusion. Note that the DNS solution shows small scale fluctuations in the $x_2$ direction due to the multi scale nature of the stream function in the $x_2$ direction but the profile in the $x_1$ direction is smooth.

\begin{figure}
\centering
\subfloat[Randomly generated stream function\label{fig:exp1a}]{\includegraphics[width=0.45\textwidth]{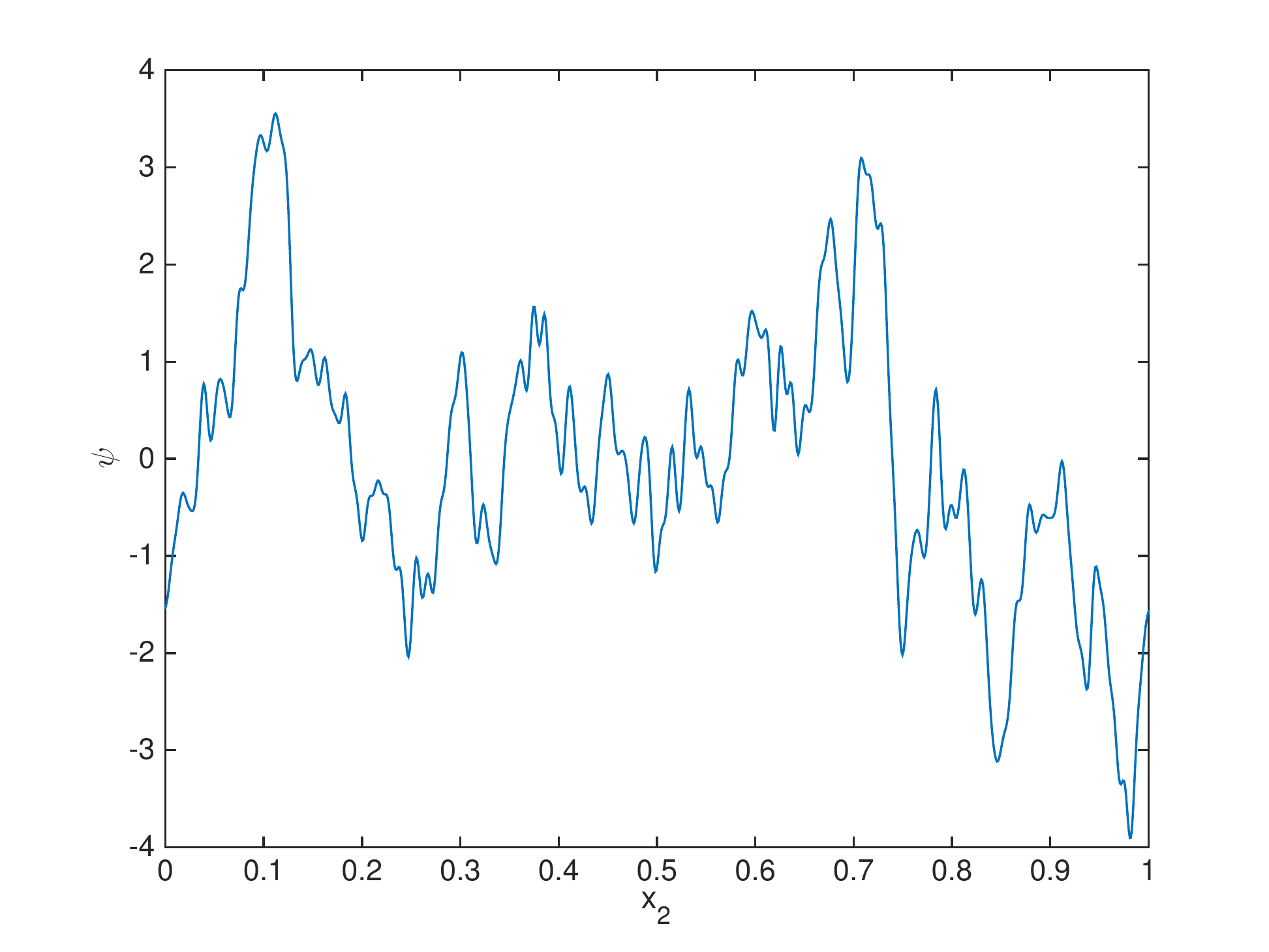}}
\subfloat[$u$ at $y=0.5$ and $t=0.1$\label{fig:exp1b}]{\includegraphics[width=0.45\textwidth]{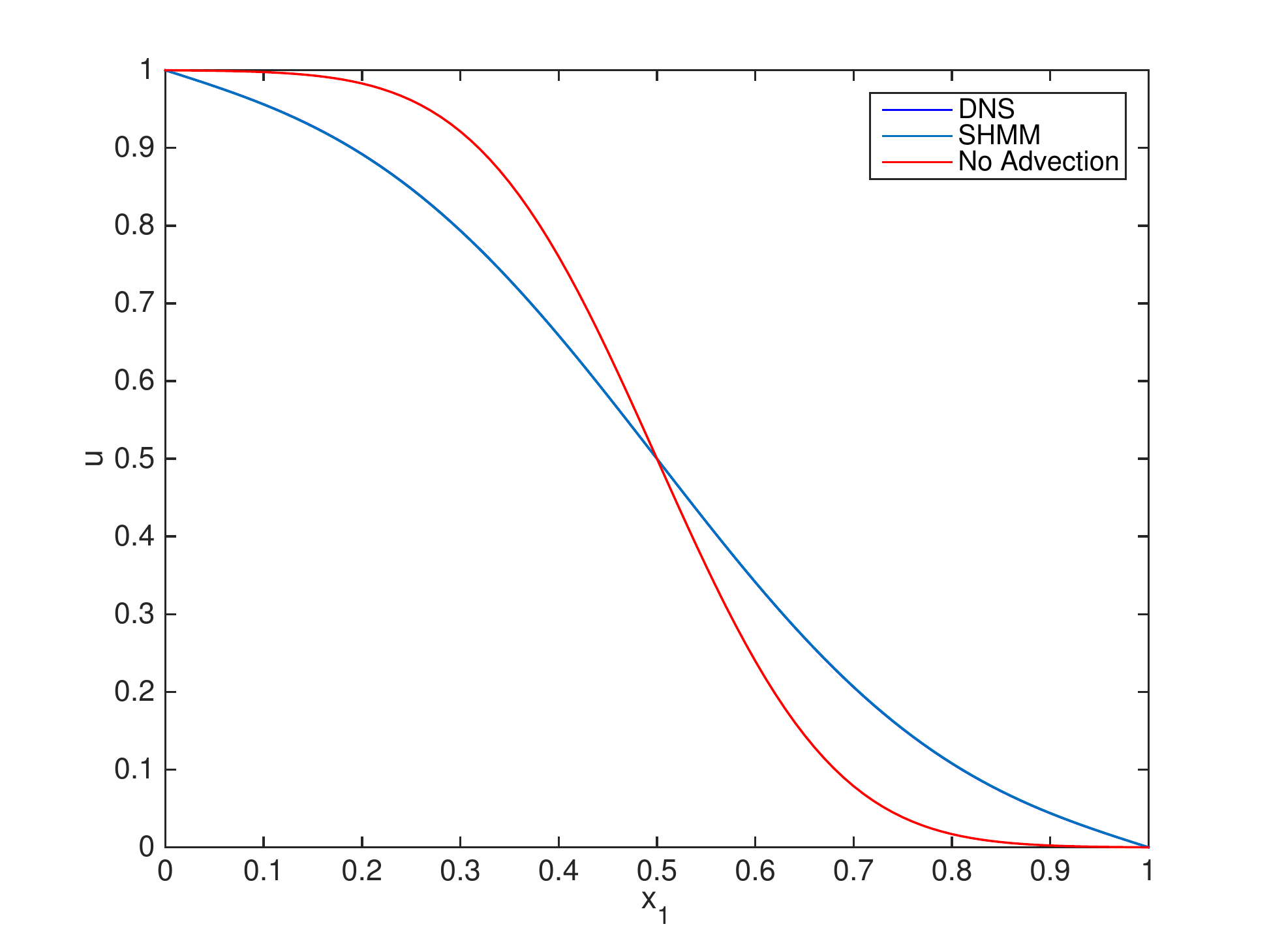}}
\caption{Randomly generated stream function and solutions by DNS, SHMM, and no advection at $y=0.5$ and $t=0.1$. The max wavenumber of the stream function is 50. SHMM solution is on the of the DNS solution.}
\end{figure}

\subsection{Stream function with two well-separated components}
The next test problem is a well-separated cellular flow which has only diagonal components in the decomposition of the stream function by SHMM,
$$\psi(\mathbf{x})=\frac{1}{3}(\phi_c(5x_1,5x_2)+\phi_c(25x_1,25x_2))$$
where
$$\phi_c(x_1,x_2)=\sin(2\pi x_1)\sin(2\pi x_2)$$
is the mean zero stream function of a cellular flow (see Figure \ref{fig:exp2a} for the contour line of the stream function). The decomposition factor here is $\alpha=5$ so that the decomposition gives two stream components, $\psi_{11}=\frac{1}{3}\phi_c(5x_1,5y_2)$ and $\psi_{22}=\phi_c(25x_1,25x_2)$ while the two off-diagonal components $\psi_{12}$ and $\psi_{21}$ are zero.

In this case, each component is periodic in the domain and well separated from each other. SHMM gives here a better approximation to the effective diffusivity than the shear flow case. See Table \ref{tab:ADVDIF:test2} for the effective diffusivity by DNS and SHMM. The relative errors of the effective diffusivity in each direction are less than $0.05\%$ which is on the order of numerical error by DNS.

\begin{table}
\begin{center}
\begin{tabular}{|c|c|c|}
\hline
&\textrm{DNS}&\textrm{SHMM}\\
\hline
\textrm{effective diffusivity}&$\begin{pmatrix}2.4801&0\\0&2.4805\end{pmatrix}$&$\begin{pmatrix}2.4798&0\\0&2.4793\end{pmatrix}$\\
\hline
\end{tabular}
\caption{Effective diffusivity from of the stream function with two separated components using DNS and SHMM}\label{tab:ADVDIF:test2}
\end{center}
\end{table}

\begin{figure}
\centering
\subfloat[Contour line of the stream function\label{fig:exp2a}]{\includegraphics[width=0.45\textwidth]{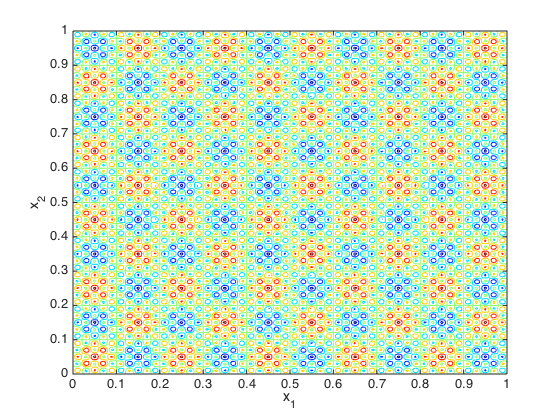}}
\subfloat[$u$ at $y=0.5$ and $t=0.1$\label{fig:exp2b}]{\includegraphics[width=0.45\textwidth]{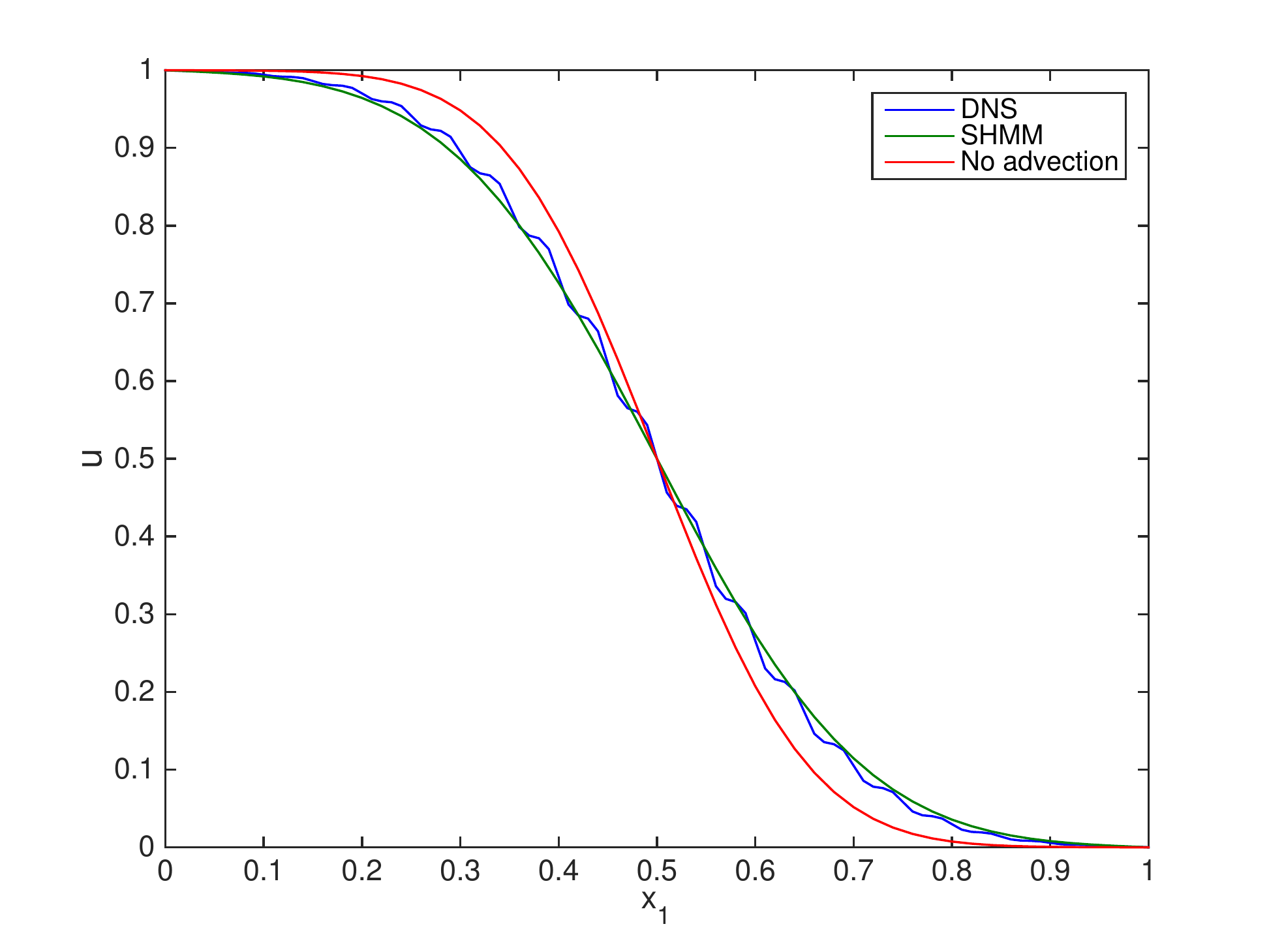}}
\caption{Stream function of two separated periodic components and solutions by DNS, SHMM, and no advection at $y=0.5$ and $t=0.1$.}
\end{figure}

\subsection{Stream function with off-diagonal components}
This test is to check the main source of SHMM errors, the partially resolved interactions between fine scales. The stream function in this test has diagonal and off-diagonal parts
$$\psi(x_1,x_2)=\frac{1}{5}\left(\phi_c(5x_1,5x_2)+\phi_c(5x_1,45x_2)+\phi_c(45x_1,5x_2)+\phi_c(50x_1,50x_2)\right)$$
where each component is of cellular type (see Figure \ref{fig:exp3a} for the contour line of the stream function). If we use the same decomposition factor $\alpha=5$ as before, this stream function contains four components
\begin{eqnarray}
\nonumber\psi_{11}&=&\frac{1}{5}\phi_c(5x_1,5x_2),\\
\nonumber\psi_{12}&=&\frac{1}{5}\phi_c(5x_1,45x_2),\\
\nonumber\psi_{21}&=&\frac{1}{5}\phi_c(45x_1,5x_2),\\
\nonumber\psi_{22}&=&\frac{1}{5}\phi_c(50x_1,50x_2).
\end{eqnarray}
Each component has non-trivial connections with each other except the pair $\psi_{11}$ and $\psi_{22}$ which we numerically verified  in the previous numerical test. The two diagonal parts $\psi_{11}$ and $\psi_{22}$ are periodic in each local domain $\Omega_{1}$ and $\Omega_{2}$ and the off-diagonal parts utilize analytic formula for effective diffusivity. Thus, the main source of error of this test problem is the partially resolved interactions between the off-diagonal and the diagonal $\psi_{22}$ components.
\begin{figure}
\centering
\subfloat[Contour line of the stream function\label{fig:exp3a}]{\includegraphics[width=0.45\textwidth]{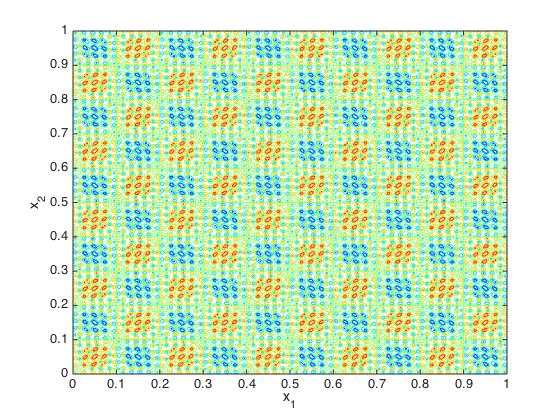}}
\subfloat[$u$ at $y=0.5$ and $t=0.1$\label{fig:exp3b}]{\includegraphics[width=0.45\textwidth]{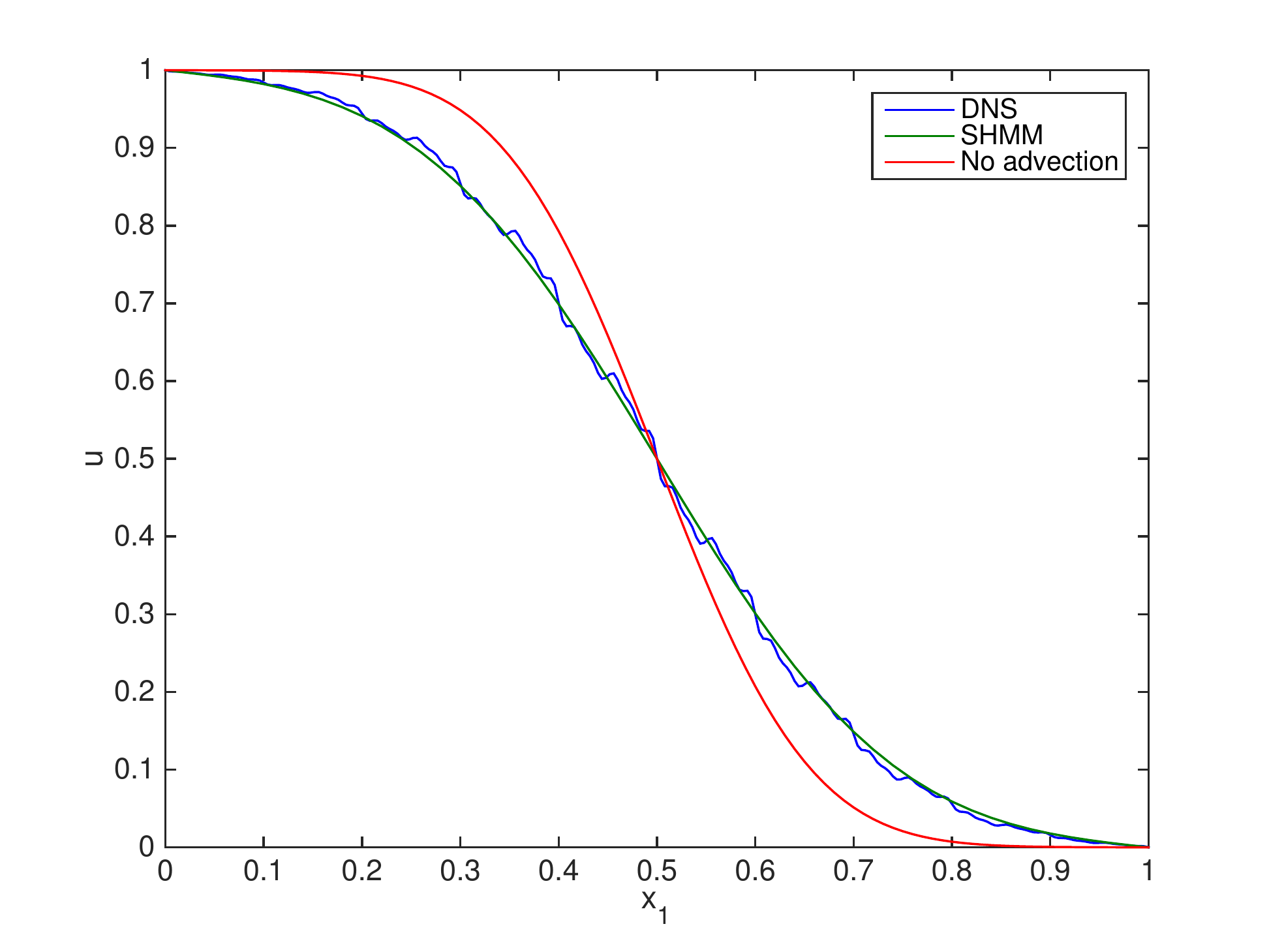}}
\caption{Stream function with off-diagonal components and solutions by DNS, SHMM, and no advection at $y=0.5$ and $t=0.1$.}\label{fig:exp3}
\end{figure}

 The numerically calculated effective diffusivity is in Table \ref{tab:ADVDIF:test3}. The effective diffusivity by SHMM in each direction have relative errors of order 5\%. Figure \ref{fig:exp3b} shows the solutions by DNS, SHMM and DNS without advection at $y=0.5$ and $t=0.1$. As in the previous test with two separated cellular flow case, SHMM solution is within the fluctuations of the DNS solution.
 \begin{table}
\begin{center}
\begin{tabular}{|c|c|c|}
\hline
&\textrm{DNS}&\textrm{SHMM}\\
\hline
\textrm{effective diffusivity}&$\begin{pmatrix}2.7592&0\\0&2.7593\end{pmatrix}$&$\begin{pmatrix}2.5981&0\\0&2.6439\end{pmatrix}$\\
\hline
\end{tabular}
\caption{Effective diffusivity of the stream function with off-diagonal components using DNS and SHMM}\label{tab:ADVDIF:test3}
\end{center}
\end{table}

\subsection{Continuous Spectrum Case}
The last numerical test has a stream function with continuous spectrum. The stream function is randomly generated in $\Omega$ and wavenumbers larger than 50 are truncated. Further the Fourier coefficients are scaled at a rate of $\mathcal{O}(\frac{1}{k^3})$ which corresponds to the inertial range spectrum of two dimensional turbulence flow. Note that the dissipative range has a much steeper spectrum. We choose the inertial range scaling so that the small scales have more apparent effect.
Figure \ref{fig:exp4a} (c)-(f) show the decomposed four components of the stream function using a decomposition factor $\alpha=5$; $\psi_{11}$ has only large scale variation while the off-diagonal parts $\psi_{12}$ and $\psi_{21}$ are similar to shear flows; the smallest component $\psi_{22}$ shows no significant large scale variation.

\begin{figure}[h!]
\centering
\subfloat[$\psi$\label{fig:exp4a}]{\includegraphics[width=0.45\textwidth]{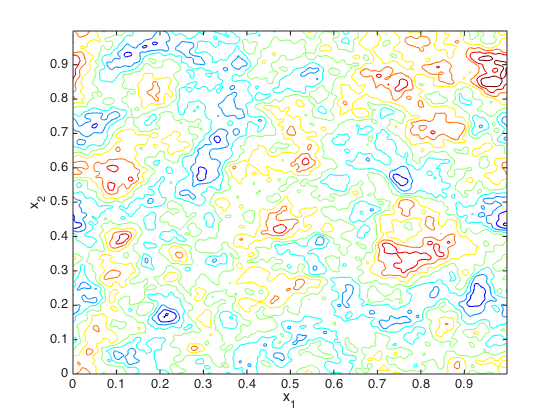}}
\subfloat[$u$ at $y=0.5$ and $t=0.1$\label{fig:exp4b}]{\includegraphics[width=0.45\textwidth]{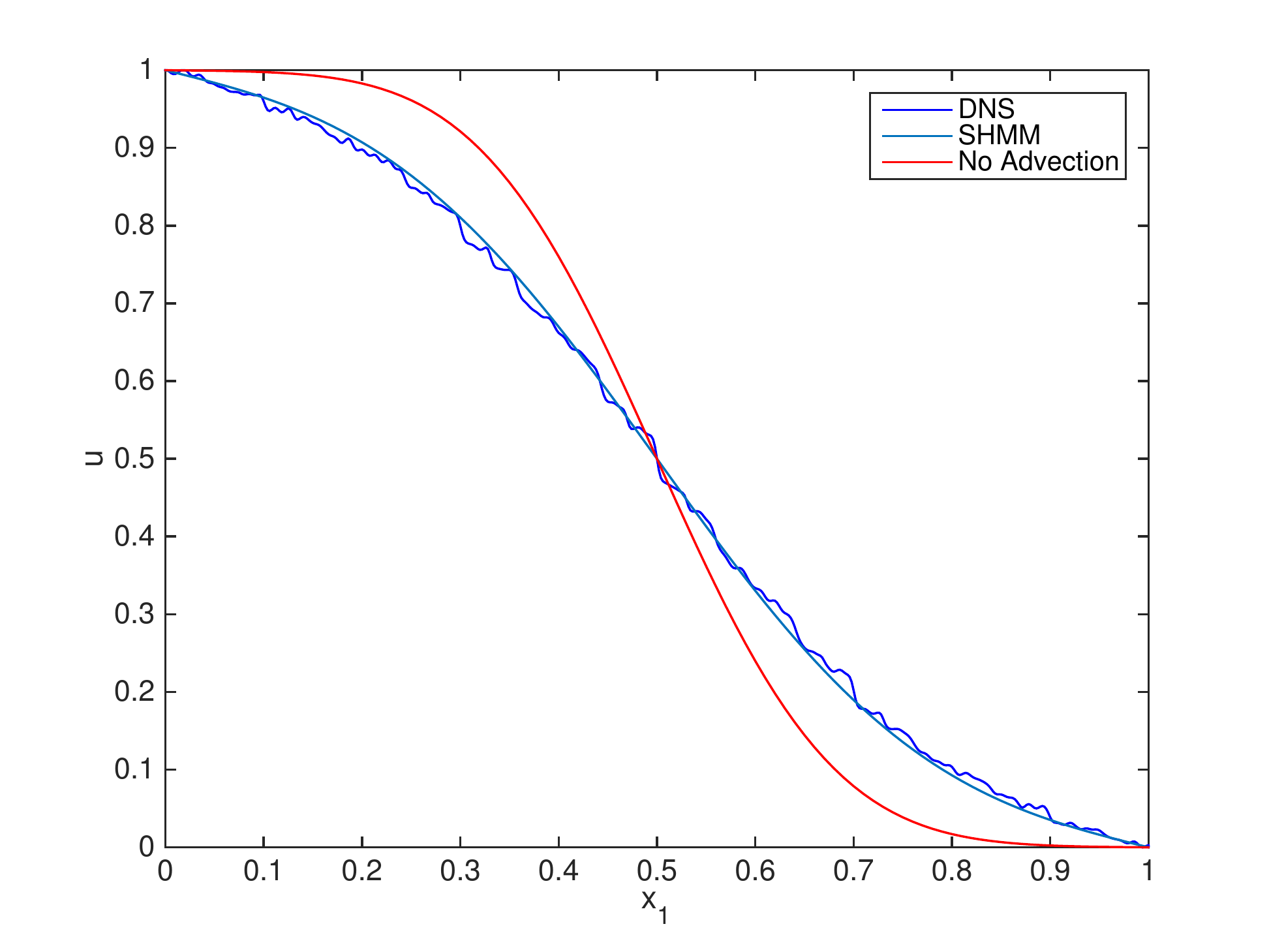}}\\
\subfloat[$\psi_{11}$\label{fig:exp4c}]{\includegraphics[width=0.45\textwidth]{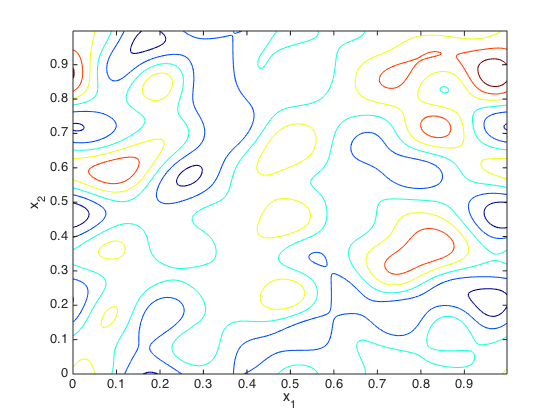}}
\subfloat[$\psi_{12}$\label{fig:exp4d}]{\includegraphics[width=0.45\textwidth]{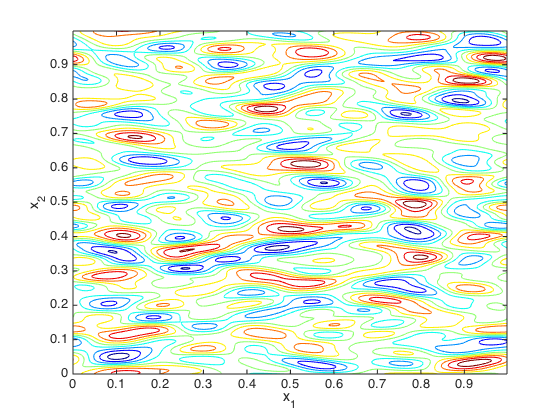}}\\
\subfloat[$\psi_{21}$\label{fig:exp4e}]{\includegraphics[width=0.45\textwidth]{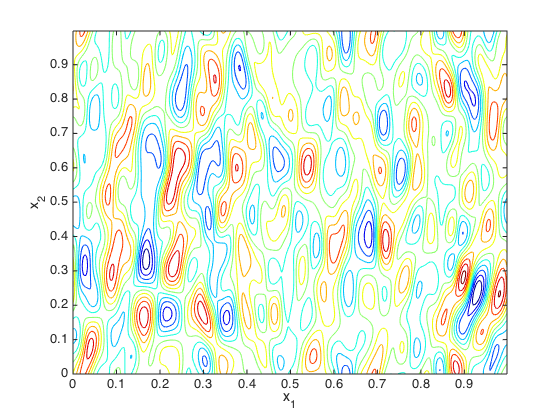}}
\subfloat[$\psi_{22}$\label{fig:exp4f}]{\includegraphics[width=0.45\textwidth]{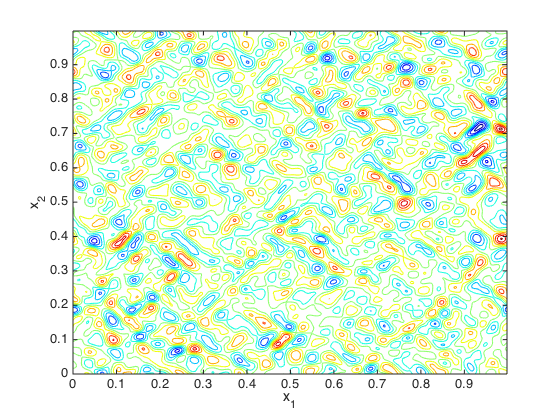}}
\caption{(a) Randomly generated stream function with continuous spectrum and (b) solutions by DNS, SHMM, and no advection. (c)-(f) Decomposition of the stream function $\psi$ into four parts using a decomposition factor $\alpha=5$}\label{fig:ADVDIF:contstream}
\end{figure}

Figure \ref{fig:exp4b} shows the solutions by DNS and SHMM with no advection at $y=0.5$ and $t=0.1$. The SHMM solution is within the fluctuations of the DNS solution without deviating from the DNS solution as in the previous tests. The directly calculated effective diffusivity and the SHMM diffusivity are given in Table \ref{tab:ADVDIF:test4} for reference. There are many possible sources of errors in SHMM such as partially resolved interactions between fine scales of the decomposed components and non-periodicity of the components, the relative error of the effective diffusivity is about 4\%.

\begin{table}[h!]
\begin{center}
\begin{tabular}{|c|c|c|}
\hline
&\textrm{DNS}&\textrm{SHMM}\\
\hline
\textrm{effective diffusivity}&$\begin{pmatrix}2.4441&0\\0&2.7213\end{pmatrix}$&$\begin{pmatrix}2.3509&0\\0&2.7182\end{pmatrix}$\\
\hline
\end{tabular}
\caption{Effective diffusivity from a velocity field with a continuum of scales using DNS and SHMM.}\label{tab:ADVDIF:test4}
\end{center}
\end{table}

\section{Conclusions}\label{CONCLUSION}
In this paper, we have developed the seamless heterogeneous multiscale method (SHMM) to numerically approximate the macroscopic behavior of the passive advection-diffusion equation for 2D steady incompressible flows with a wide range of spatial scales. SHMM imposes scale decomposition of stream functions in the Fourier domain and uses a hierarchy of local grids. The computational complexity only increases linearly as the umber of decomposition level increases. For each level, it uses HMM and closed analytic formulas for shear flows to get the effective diffusivity from each component of the velocity fields. This process is iterated up to the coarse scale of interest. The method is numerically tested and verified to capture the effective diffusivity of various velocity fields with and without scale separation. 

The method can also be extended to the velocity fields with multiple temporal scales using decomposition of velocity fields in time and utilizing multiscale time integrators such as \cite{VSHMM}. An application of SHMM to multi-spatiotemporal velocity fields is reported in \cite{3scVSHMM}.

It would be natural to speculate if the technique presented here can be extended to the fully nonlinear problem of simulating turbulence. There are many obstacles including intermittency. It should however be possible to describe some effects of homogeneous turbulent fluctuations on the mean flow even in the case of full Navier-Stokes. Compare here the effect the $k$ and $\epsilon$ terms have on the mean flow in a $k$-$\epsilon$ model. Also, a recent study \cite{HMMlikeTub} shows great success in simulating turbulence in channel flows using a HMM style combination of coarse scale large-eddy simulations and an array of non-sapce-filling quasi-direct numerical simulations next to the boundary. We plan to study if SHMM can play a role in this context.

\section*{Acknowledgements}
This work has greatly benefited from discussions with Andrew J. Majda, Weinan E and Richard Tsai. The research was partially supported by NSF grant DMS-1217203.



\bibliography{SHMM}{}
\bibliographystyle{elsarticle-num}


\end{document}